\documentclass[a4paper,11pt,reqno]{amsart}
\usepackage{amsfonts}
\usepackage{amsmath}
\usepackage{amssymb}

\newtheorem{theorem}{Theorem}

\newtheorem{corollary}{Corollary}
\newtheorem{remark}{Remark}
\newtheorem{lemma}{Lemma}

\def\be{\begin{equation}}
\def\ee{\end{equation}}
\def\ben{\begin{displaymath}}
\def\een{\end{displaymath}}
\def\baa{\begin{eqnarray}}
\def\eaa{\end{eqnarray}}

\def\ba{\begin{array}}
\def\ea{\end{array}}
\def\la{\label}

\makeatletter
\@addtoreset{equation}{section}
\makeatother

\def\C{{\mathbb C}}
\def\Z{{\mathbb Z}}
\def\Q{{\mathbb Q}}
\def\H{{\mathcal H}}
\def\Ht{\tilde{{\mathcal H}}}
\def\Hc{\check{{\mathcal H}}}
\def\Hb{\overline{{\mathcal H}}}
\def\M{{\mathcal M}}
\def\Mb{\overline{{\mathcal M}}}
\def\Pl{\mathbb{P}^1}
\def\2x2{{\left(\!\!\begin{array}{cc}a&b\\c&d\\\end{array}\!\!\right)}}
\def\del{\partial}
\def\f{\frac}
\def\e{\epsilon} 

\begin{document}
\title[Tau function and admissible covers]{Isomonodromic tau function on the space of admissible covers}

\author{A.~Kokotov$^1$}
\address{$^1$Department of Mathematics and Statistics, Concordia University,
1455 de Maisonneuve West, Montreal, H3G 1M8  Quebec,  Canada}
\author{D.~Korotkin$^1$}
\author{P.~Zograf$^2$}
\address{$^2$Steklov Mathematical Institute, Fontanka 27, St. Petersburg 191023 Russia}

\thanks{AK and DK were supported in part by NSERC; DK was also
supported by FQRNT and CURC; PZ was partially supported by the RFBR grant 08-01-00379-a and 
by the President of Russian Federation grant NSh-2460.2008.1.}

\begin{abstract}
The isomonodromic tau function of the Fuchsian
differential equations associated to Frobenius structures on
Hurwitz spaces can be viewed as a section of a line bundle on
the space of admissible covers. We study the asymptotic behavior
of the tau function near the boundary of this space and compute its
divisor. This yields an explicit formula for the pullback of the Hodge
class to the space of admissible covers in terms of the classes of 
compactification divisors.
\end{abstract}

\maketitle

\tableofcontents

\section{Introduction}

The space of admissible covers is a natural compactificaion of the Hurwitz space of smooth branched covers
of the complex projective line $\Pl$, or, equivalently, meromorphic functions on complex algebraic curves, of given
degree and genus. This space was first introduced by J.~Harris and D.~Mumford and appeared to be quite useful in computing
the Kodaira dimension of the moduli space of stable curves \cite{HMu}. Lately this space has attracted a major attention,
mainly in connection with Gromov-Witten theory, quantum cohomology, Hurwitz numbers, Hodge integrals, etc. (The literature
on this subject is abundant, and it is not possible to give even a very brief review here.) 

On the other hand, Hurwitz spaces appear naturally in relationship with the Riemann-Hilbert problem, and carry a natural
Frobenius structure \cite{D}. The tau function for the corresponding isomonodromic deformations can be written explicitly
in terms of the theta function and the prime form on the covering complex curve \cite{KK1}.

In this paper we study the asymptotic behavior of the isomonodromic tau function near the boundary of the Hurwitz space
given by nodal admissible covers, and explicitly compute its divisor. 
More precisely, a power of the tau function
corrected by a power the Vandermonde determinant of the critical values of the branched
cover descends to a holomorphic section of (the pullback of) the Hodge bundle on the
Hurwitz space. Moreover this section extends to a meromorphic section of
the Hodge bundle on the compactification of the Hurwitz space by
admissible covers.
This allows us to express (the pullback of) the 
Hodge class on the space of admissible covers as a linear combination of boundary divisors (in small genera this also 
gives a non-trivial relation between the boundary divisors). 

The paper is organized as follows. In Section 2 we define the isomonodromic tau function, give an explicit formula for it 
(Theorem 1), 
study its transformation properties and interprete it as a holomorphic section of a line bundle on the Hurwitz space.
Section 3 contains the main results of the paper: an asymptotic formula for the tau function near the boundary of the
space of admissible covers (Theorem 2), and a formula for the Hodge class in terms of the classes of boundary divisors (Theorem 3).  
The special cases of the latter include a formula of Cornalba-Harris for the Hodge class on the hyperelliptic locus \cite{CH}, and
a relation of Lando-Zvonkine between the compactification divisors in Hurwitz spaces of genus 0 branched covers \cite{LZ}. 

\section{Isomonodromic tau function}

\subsection{Hurwitz spaces}

Let $C$ be a smooth complex algebraic curve of genus $g$, and let
$f$ be a meromorphic function on $C$ of degree $d>0$. We can think
of $f$ as a holomorhic branched cover $f: C\rightarrow \Pl$
over the projective line $\Pl$. We call a meromorphic function (or
a branched cover) {\em generic} if it has only simple critical
values (branch points). For a generic $f$ the number of branch points is $n=2g+2d-2$, we 
denote them by $z_1,\dots,z_n\in\Pl$ and always assume that they are {\em ordered}. 

Two meromorphic functions $f_1:C_1\rightarrow\Pl$ and $f_2:C_2\rightarrow\Pl$
are called {\em strongly equivalent} (or simply {\em equivalent}), if there exists
an isomorphism $h: C_1\rightarrow C_2$ such that $f_1=f_2\circ h$, and
{\em weakly equivalent}, if there exist isomorphisms $h: C_1\rightarrow C_2$
and $\gamma: \Pl\rightarrow\Pl$ such that $\gamma\circ f_1\!=\!f_2\circ h$.
In addition to that we will also consider an equivalence relation for
meromorhic functions on Torelli marked curves. A {\em Torelli marking} is a choice
of symplectic basis $\alpha=\{a_i, b_i\}_{i=1}^g$ in the first homology group
$H_1(C)$ of $C$. A curve $C$ together with a symplectic basis $\alpha$ will be
denoted by $C^\alpha$. We say that two meromorphic functions on Torelli marked curves are
{\em Torelli equivalent}, if for Torelli marked curves
$C_1^{\alpha_1}, C_2^{\alpha_2}$ there exist an isomorphism $h: C_1\rightarrow C_2$ such that
$f_1=f_2\circ h$ and $h_*(\alpha_1)=\alpha_2$ elementwise.

For any fixed $g\geq 0$ and $d>0$ consider the space of all generic meromorphic
functions of degree $d$ on all smooth genus $g$ curves.
Denote by $\H_{g,d},\,\Ht_{g,d},\Hc_{g,d}$ the moduli spaces (called
{\em Hurwitz spaces}) defined by the weak, strong and Torelli equivalence
relations respectively (the latter requires the curves to be Torelli marked).
All three spaces are non-compact complex manifolds. The last two spaces
have dimension $n=2g+2d-2$ and the branch points $z_1,\dots,z_n$ provide
a system of local coordinates for both of them.
The group $PSL(2,\C)$ acts freely on $\Ht_{g,d}$ and $\Hc_{g,d}$ by linear
fractional transformations: for $\gamma=\2x2 \in PSL(2,\C)$ we have
$\gamma\circ f=\frac{af+b}{cf+d}$, so that, in particular,
$\H_{g,d}=\Ht_{g,d}/PSL(2,\C)$. In addition, the symplectic group
$Sp(2g,\Z)$ acts on $\Hc_{g,d}$ by changing Torelli marking, and
$\Ht_{g,d}=\Hc_{g,d}/Sp(2g,\Z)$. The actions of $PSL(2,\C)$ and $Sp(2g,\Z)$ on
$\Hc_{g,d}$ clearly commute.

In the sequel we will also deal with meromorphic functions (branched covers)
that have one fixed value, either regular at $z=\infty$, or  degenerate critical
of type $\mu=[m_1,\dots,m_r]$ at any $z\in \Pl$
($m_i>0$ are the ramification degrees of
the points in $f^{-1}(z),\;m_1+\dots+m_r=d$), with all other branch points being 
simple and finite (the number of these critical
values is $n(\mu)=2g+d+r-2$). The Hurwitz spaces of such functions defined 
modulo the weak (while keeping $z$ fixed), strong and Torelli equivalence relations we denote by 
$\H_{g,d}(z,\mu),\,\Ht_{g,d}(z,\mu)$ and $\Hc_{g,d}(z,\mu)$ respectively. The dimension
of the last two ones is $n(\mu)=2g+d+r-2$, and the simple branch points $z_1,\dots, z_{n(\mu)}$ serve as 
local coordinates for them as well. In particular, $\Ht_{g,d}(\infty, 1^d)$ and $\Hc_{g,d}(\infty,1^d)$
are open dense subsets of the Hurwitz spaces $\Ht_{g,d}$ and $\Hc_{g,d}$ respectively.

\subsection{Definition of the tau function}

For a Torelli marked curve $C^\alpha$, denote by $B(x,y)$ the 
{\em Bergman bidifferential}, that is, the unique symmetric meromorphic bidifferential 
on $C\times C$ with a quadratic pole of biresidue 1 on the diagonal and zero $a$-periods
(the details on meromorphic bidifferentials and the associated projective connections
can be found, e.g., in \cite{F} or \cite{T}). 
The $b$-periods of the Bergman bidifferential $B(x,y)$
\be
\omega_i=\int_{b_i}B(\cdot\,, y)dy
\ee
are the {\em normalized holomorphic differentials} on $C^\alpha$, that is,
\be
\int_{a_j}\omega_i=\delta_{ij},\quad\quad \int_{b_j}\omega_i=\Omega_{ij},\quad\quad i,j=1,\dots, g,
\ee
where the matrix $\Omega=\{\Omega_{ij}\}_{i,j=1}^g$ is the {\em period matrix} of $C^\alpha$.
In terms of local parameters $\zeta(x),\zeta(y)$ near the diagonal $\{x=y\}\in C\times C$, the bidifferential $B(x,y)$
has the following Laurent series expasion in  $\zeta(y)$  at the point  $\zeta(x)$
\be
B(x,y)=\left(\f{1}{(\zeta(x)-\zeta(y))^2}+\f{S_B(\zeta(x))}{6}
+O((\zeta(x)-\zeta(y))^2)\right)d\zeta(x) d\zeta(y),
\la{asW}
\ee
where $S_B$ is a projective connection on $C$ called the {\em Bergman projective connection}.
The latter means that $S_B$ transforms under the change $\zeta=\zeta(w)$ of the local parameter 
by the rule $S_B(w)=S_B(\zeta(w))\zeta'(w)^2+S_\zeta$, where 
$S_\zeta=\frac{\zeta'''}{\zeta'}-\frac{3}{2}\left(\frac{\zeta''}{\zeta'}\right)^2$ is the 
{\em Schwarzian derivative} of $\zeta(w)$ with respect to $w$.

Now consider the Schwarzian derivative $S_f=\frac{f'''}{f'}-\frac{3}{2}\left(\frac{f''}{f'}\right)^2$
of a meromorphic function $f:C\to\Pl$ with respect to a local parameter $\zeta$ on $C$. 
This is a meromorphic projective connection on $C$, so that the
difference $S_B-S_f$ is a meromorphic quadratic differential. 
Take the trivial line bundle on the Hurwitz space $\Hc_{g,d}(z,\mu)$ and consider the connection
\be
d_B=d+4\,\sum_{i=1}^{n(\mu)} \left({\rm Res}_{\,x_i}\f{S_B-S_f}{d f}\right) dz_i,
\la{Bercon}
\ee
where the sum is taken over all simple finite branch points $z_i$ of $f$,
and $x_i\in C$ are the corresponding critical points. Rauch's formulas imply that
this connection is flat (cf. \cite {KK1}). The tau function $\tau(C^\alpha, f)$ 
is locally defined as a horizontal
(covariant constant) section of the trivial line bundle on $\Hc_{g,d}(z,\mu)$ with respect to $d_B$:
\footnote{This tau function is the 24th power of the Bergman tau function studied in \cite {KK1} 
and the -48th power of the isomonodromic tau function
of the Frobenius manifold structure on
Hurwitz space introduced by Dubrovin \cite{D}. Our present definition
is more appropriate in the context of admissible covers.}
\be
\frac{\del \log \tau(C^\alpha,f)}{\del z_i}=
-4\,{\rm Res}_{\,x_i}\f{S_B-S_f}{d f},\ \
i=1, \dots, n.
\la{tau}
\ee

Let us now recall an explicit formula for the tau function $\tau(C^\alpha, f)$ derived in \cite {KK1}.
Take a nonsingular odd theta characteristic $\delta$ and consider the corresponding theta function
$\theta[\delta](v; \Omega)$, where $v=(v_1,\dots,v_g)\in \C^g$. Put 
$$\omega_\delta=\sum_{i=1}^g\frac{\del\theta[\delta]}{\del v_i}\left(0; \Omega\right)\,\omega_i\,.$$
All zeroes of the holomorphic 1-differential $\omega_\delta$ have even multiplicities, and 
$\sqrt{\omega_\delta}$ is a well-defined holomorphic spinor on $C$. Following Fay \cite{F}, consider 
the prime form \footnote{The prime form $E(x,y)$ is the canonical section of the line bundle on 
$C\times C$ associated with the diagonal divisor $\{x=y\}\subset C\times C$.}
\be
E(x,y)=\frac{\theta[\delta]\left(\int_x^y\omega_1,\dots,\int_x^y\omega_g; \Omega\right)}
{\sqrt{\omega_\delta}(x)\sqrt{\omega_\delta}(y)}.
\ee
To make the integrals uniquely defined, we fix $2g$ simple closed loops in the homology classes $a_i,b_i$
that cut $C$ into a connected domain, and pick the integration paths that do not intersect the cuts. 
The sign of the square root is chosen so that 
$E(x,y)=\frac{\zeta(y)-\zeta(x)}{\sqrt{d\zeta}(x)\sqrt{d\zeta}(y)}(1+O((\zeta(y)-\zeta(x))^2))$ as $y\rightarrow x$,
where $\zeta$ is a local parameter such that $d\zeta=\omega_\delta$.

We introduce local coordinates on $C$ that we call {\em natural} (or {\em distinguished}) with respect to $f$.
Consider the divisor $(df)=\sum_k d_k\,p_k,\;p_k\in C,\;d_k\in\Z,d_k\neq 0,$ of the meromorphic differential $df$.
We take $\zeta=f(x)$ as a local coordinate on $C-\bigcup_k p_k$, and  
\be
\zeta_k=\begin{cases}\left(f(x)-f(p_k)\right)^{\frac{1}{d_k+1}}&\text{if $d_k>0$,}\\
f(x)^{\frac{1}{d_k+1}}&\text{if $d_k<0$,}\end{cases}\label{np}
\ee
near $p_k\in C$. In terms of these coordinates we have 
$E(x,y)=\frac{E(\zeta(x),\zeta(y))}{\sqrt{d\zeta}(x)\sqrt{d\zeta}(y)}$, and we define
\begin{eqnarray*}
E(\zeta,p_k)&=&\lim_{y\rightarrow p_k}E(\zeta(x),\zeta(y))\sqrt{\frac{d\zeta_k}{d\zeta}}(y),\\
E(p_k,p_l)&=&\lim_{\stackrel{\scriptstyle x\rightarrow p_k}{y\rightarrow p_l}}E(\zeta(x),\zeta(y))
\sqrt{\frac{d\zeta_k}{d\zeta}}(x)\sqrt{\frac{d\zeta_l}{d\zeta}}(y)\,.
\end{eqnarray*}
Let ${\mathcal A}^x$ be the Abel map with the basepoint $x$, and let $K^x=(K^x_1,\dots,K^x_g)$ be the vector of Riemann constants
\be
K^x_i=\frac{1}{2}+\frac{1}{2}\Omega_{ii}-\sum_{j\neq i}\int_{a_i}\left(\omega_i(y)\int_x^y\omega_j\right)dy
\label{rc}\ee
(as above, we assume that the integration paths do not intersect the cuts on $C$).
Then we have ${\mathcal A}^x((df))+2K^x=\Omega Z+Z'$ for some $Z,Z'\in\Z^g$ . Now put
\be
\tau(C^\alpha, f)=
\f{\left(\left.\left(\sum_{i=1}^g\omega_i(\zeta)\frac{\partial}{\partial v_i}\right)^g
\theta(v;\Omega)\right|_{v=K^{\zeta}}\right)^{16}}{e^{4\pi\sqrt{-1}\langle\Omega Z+4K^\zeta,Z\rangle}\;W(\zeta)^{16}}
\frac{\prod_{k<l}E(p_k,p_l)^{4d_k d_l}}
{\prod_k E(\zeta,p_k)^{8(g-1)d_k}}\;.
\label{taukk}
\ee
Here $\theta(v;\Omega)=\theta[0](v;\Omega)$ is the Riemann theta function, $v=(v_1,\dots,v_g)\in\C^g$,
and $W$ is the Wronskian of the normalized holomorphic differentials $\omega_1,\dots,\omega_g$ on 
$C^\alpha$. \footnote{The expression 
${\mathcal C}(\zeta)=\frac{1}{W(\zeta)}\left.\left(\sum_{i=1}^g\omega_i(\zeta)\frac{\partial}{\partial v_i}\right)^g
\theta(v;\Omega)\right|_{v=K^{\zeta}}$ first appeared in \cite{F} in a different context.}

\begin{theorem}{\rm (cf. \cite {KK1})} Let $\tau(C^\alpha, f)$ be given by formula (\ref{taukk}). Then\\
(i) $\tau(C^\alpha, f)$ does not depend on either $\zeta$ or the choice of the cuts
in the homology classes $a_i, b_i$;\\
(ii) $\tau(C^\alpha, f)$ is a nowhere vanishing holomorphic function on the Hurwitz space $\Hc_{g,d}(\infty,1^d)$
of generic meromorphic functions with only finite simple branch points, whereas on the Hurwitz spaces 
$\Hc_{g,d}(z,\mu)$ with non-trivial $\mu$ it is defined 
locally up to a root of unity and may depend on the choice of the parameters $\zeta_k$;\\
(iii) $\tau(C^\alpha, f)$ is an isomonodromic tau function, that is, a solution of (\ref{tau}).
\end{theorem}

\subsection{Tau function as a section of a line bundle}

We start with describing the behavior of the tau function under the linear fractional
transformations of $f$ and changing of Torelli marking on $C$. Unfortunately, $\tau(C^\alpha, f)$ is smooth
only on $\Hc_{g,d}(\infty, 1^d)$ and becomes singular on the entire Hurwitz space $\Hc_{g,d}$. 
To overcome this difficulty, denote by  $V(z_1,\dots, z_n)= \prod_{i<j} (z_i-z_j)$ the Vandermonde determinant of the critical 
values $z_1,\dots, z_n$ of $f$, and put
\be
\check{\eta}= \tau^{n-1}V^{-6}
\la{eta}
\ee
\begin{lemma}
The function $\check{\eta}=\check{\eta}(C^\alpha, f)$ extends to the Hurwitz space
$\Hc_{g,d}$ as a nowhere vanishing holomorphic function and is invariant with respect to the natural action of $PSL(2,\C)$ on $\Hc_{g,d}$.
\end{lemma}

\begin{proof}
Take $\gamma=\2x2\in PSL(2,\C)$ and denote by $z^\gamma_i:=\frac{a z_i+b}{c z_i +d}$ the branch points of the function 
$f^\gamma=\gamma\circ f$.

According to (\ref{tau}), we have
\be
\frac{\del\log\tau(C^\alpha, f^\gamma)}{\del z_i^\gamma}=-4\, {\rm Res}_{\,x_i}\frac{S_B-S_{f^\gamma}}{d f^\gamma}=
-4\, {\rm Res}_{\,x_i}\frac{S_B-S_{f}}{d f^\gamma}\,,
\la{projconSSm}
\ee
since $S_{f^\gamma}=S_f$ by the property of the Schwarzian derivative. 
Moreover, we have $d z_i^\gamma/dz_j= (cz_i+d)^{-2}$,
and $d f^\gamma/df = (cf+d)^{-2}$, so that
$$
\frac{\del\log\tau(C^\alpha, f^\gamma)}{\del z_i}=\f{-4}{(c z_i+d)^2}\,{\rm Res}_{\,x_i}\left((cf+d)^2\frac{S_B-S_{f}}{d f}\right)\;.
$$
In terms of the natural local parameter $\zeta_i(x)=\sqrt{f(x)-z_i}$ near the critical point $p_i \in f^{-1}(z_i)$ this gives
$f=\zeta_i^2+z_i$, $df=2\zeta_i d\zeta_i$ and $S_f= -3/2\zeta_i^{-2}$.  Therefore, 
\begin{eqnarray}
&&\f{\del\log\tau(C^\alpha, f^\gamma)}{\del z_i}\nonumber \\
&&=-4\,{\rm Res}_{\,x_i}\f{S_B-S_f}{df}
-\f{3}{(cz_i+d)^2}\,{\rm Res}_{\,\zeta_i=0}\f{(c \zeta_i^2+c z_i+d)^2d\zeta_i}{\zeta_i^3}\nonumber \\
&&=\f{\del\log\tau(C^\alpha, f)}{\del z_i}-6\,\f{c}{c z_i+d}\;.
\la{transtau}
\end{eqnarray}
On the other hand, a simple computation shows that
\be
\frac{\del\log V^\gamma}{\del z_i}=\frac{\del\log V}{\del z_i}- (n-1)\f{c}{c z_i+d},
\la{transvan}
\ee
where $V^\gamma=\prod_{i< j}(z_i^\gamma-z_j^\gamma)$.

From (\ref{transtau}) and (\ref{transvan}) it follows that 
$\check{\eta}(C^\alpha, f^\gamma)=\chi(\gamma)\check{\eta}(C^\alpha, f)$,
where $\chi(\gamma)$ is a $\C^*$-representation of $PSL(2,\C)$ and
hence $\chi(\gamma)=1$ identically.
\end{proof}

The next lemma describes the behavior of the tau function under the natural action of $\C^*=\C-\{\infty\}$ 
on the space $\Hc_{g,d}(\infty, \mu),\,\mu=(m_1,\dots,m_r),$
($\C^*$ acts on meromorphic functions by multiplication, thus leaving $\infty$ fixed).  
  
\begin{lemma}
The tau function $\tau(C^\alpha, f)$ on the Hurwitz space $\Hc_{g,d}(\infty, \mu)$ has the property
\be
\tau(C^\alpha, \e f)=\e^{3n(\mu) - 2d+2\sum_{i=1}^r 1/m_i}\; \tau(C^\alpha, f)
\la{taue}
\ee
for any $\e\in\C^*$, where $n(\mu)=2g+d+r-2$ is the number of simple finite branch points of $f$.
\end{lemma}

\begin{remark}
{\rm In the case $\mu=1^d$ this is a consequence of the previous lemma
for} $\gamma={\rm diag}\,(\e^{1/2},\e^{-1/2})$.
\end{remark}

\begin{proof} 
It is easy to see that the difference between the explicit expressions for
$\tau(C^\alpha, f)$ and $\tau(C^\alpha, \e f)$ in (\ref{taukk}) comes from the different choice
of natural local parameters $\zeta$ on $C-\cup_k p_k$ and $\zeta_k$ at the points $p_k$ of the divisor $(df)$. 
Clearly, $\zeta^\e=\e\zeta$ and, according to (\ref{np}), 
$$\zeta_k^\e=\begin{cases}\e^{\frac{1}{d_k+1}}\left(f(x)-f(p_k)\right)^{\frac{1}{d_k+1}}&\text{if $d_k>0$,}\\
\e^{\frac{1}{d_k+1}}f(x)^{\frac{1}{d_k+1}}&\text{if $d_k<0$.}\end{cases}$$
Moreover, $d_k=1$ for all zeroes of $df$, and and $d_k=-m_i-1,\,i=1,\dots,r,$ for the poles of $df$.
Substituting these parameters $\zeta_k^\e$ into (\ref{taukk}), we get Eq. (\ref{taue}).
\end{proof}

\begin{lemma}
On the Hurwitz space $\Hc_{g,d}(\infty, \mu)$ we have the identity
\be
\sum_{i=1}^{n(\mu)} z_i\, {\rm Res}_{\,x_i}\f{S_B-S_f }{d f}= -\f{3n(\mu)}{4} +\f{d}{2} -\f{1}{2}\sum_{i=1}^r \f{1}{m_i}\;.
\la{Euler1}
\ee
\end{lemma}

\begin{proof} 
The homogeneity property (\ref{taue}) implies that
$$\sum_{i=1}^{n(\mu)} z_i\f{\del}{\del z_i}\log \tau(C^\alpha, f) = 3n(\mu)-2d+2\sum_{i=1}^r \f{1}{m_i}\;.$$
This immediately yields (\ref{Euler1}) due to the definition (\ref{tau}) of the tau function.
\end{proof}

The behavior of the tau function under the change of Torelli marking of $C$ is described in the following lemma: 

\begin{lemma}
Let two canonical bases $\alpha=\{a_i,b_i\}_{i=1}^g$ and $\alpha'=\{a'_i,b'_i\}_{i=1}^g$ in $H_1(C)$
be related by $\alpha'=\sigma\alpha$, where 
\be
\sigma=
\left(\begin{array}{cc} D & C\\
B & A \end{array}\right)\in Sp(2g,\Z).
\la{symtrans}\ee
Then 
\be
\f{\tau(C^{\alpha'}, f)}{\tau(C^\alpha, f)}= 
{\rm det}(C\Omega+D)^{24}\;.
\la{tautaut}
\ee
where ${\Omega}$ is the period matrix of the Torelli marked Riemann sutface $C^\alpha$.
\end{lemma}

\begin{proof}
To establish this transformation property, we use the explicit formula (\ref{taukk}). 
According to Lemma 6 of \cite{KK2}, when $df$ has at least one simple zero one can always choose 
the cut system on $C$ in such a way that $Z=Z'=0$ in (\ref{taukk}). The change of basis $\alpha'=\sigma\alpha$ 
results in the following transformation of the prime form $E(x,y)$:
\be 
E'(x,y)= E(x,y)e^{\sqrt{-1} \pi 
v (C\Omega +D)^{-1} C v^t}
\la{transE} 
\ee 
(cf. \cite{F2}, Eq. (1.20)); here $v=(\int_x^y\omega_1,\dots,\int_x^y\omega_g)$. 
For the expression 
$${\mathcal C}(x)=\frac{1}{W(x)}
\left.\left(\sum_{i=1}^g\omega_i(x)\frac{\partial}{\partial v_i}\right)^g
\theta(v;\Omega)\right|_{v=K^x}$$
it is shown in \cite{F2}, Eq. (1.23), that 
\be
{\mathcal C}'(x)=\delta ({\rm det}(C{\Omega} +D))^{3/2} \, e^{\sqrt{-1}\pi K^x (C{\Omega}+D)^{-1}C (K^x)^t}\,  {\mathcal C}(x)\;,
\la{transC}
\ee
where $\delta$ is a root of unity of eighth degree, and $K^x$ is the vector of Riemann constants (\ref{rc}). 
Substituting these formulae into (\ref{taukk}), we obtain the statement the lemma.
\end{proof}

Denote by $\lambda$ the Hodge line bundle on the Hurwitz space $\H_{g,d}$; the fiber of $\lambda$
over the point represented by a pair $(C, f)$ is isomorphic to $\det \Omega^1_C=\wedge^g \Omega^1_C$, 
where $\Omega^1_C$ is the space of holomorphic
1-forms (abelian differentials) on $C$. The line bundle $\lambda$ has a local holomorphic section given by 
$\omega_1\wedge\dots\wedge\omega_g$, where $\omega_1,\dots,\omega_g$ is the basis of normalized abelian differentials
on a Torelli marked curve $C^\alpha$. Under the change of marking $\alpha'=\sigma\alpha$ with $\sigma\in Sp(2g,\Z)$ given by
(\ref{symtrans}), this section transforms by the rule 
$\omega'_1\wedge\dots\wedge\omega'_g=\det(C\Omega+D)^{-1}\, \omega_1\wedge\dots\wedge\omega_g$. 
Combining this with Lemmas 1 and 4 we obtain

\begin{lemma}
The function $\check{\eta}=\tau^{n-1}V^{-6}$ on $\Hc_{g,d}$ descends to a nowhere vanishing holomorphic
section $\eta$ of the line bundle $\lambda^{24(n-1)}$ on $\H_{g,d}$.
\end{lemma}

\section{Divisor of the tau function}

\subsection{The space of admissible covers}

The space of admissible covers $\Hb_{g,d}$ is a natural compactification of the Hurwitz space $\H_{g,d}$
that was introduced in \cite{HMu}. An {\em admissible cover}
is a degree $d$ regular map $f:C\rightarrow R$ of a connected genus $g$ nodal curve $C$ onto a rational nodal curve $R$ 
that is simply branched over $n=2g+2d-2$ distinct points on the smooth part of $R$ and maps nodes to nodes with the same 
ramification indices for the two branches at each node. The space of (weak equivalence classes of) 
admissible covers $\Hb_{g,d}$ has relatively simple local structure, though it is not a normal algebraic variety and
therefore not an orbifold. However, a normalization of $\Hb_{g,d}$ is smooth, cf. \cite{ACV}, \cite{I}.

The space $\Hb_{g,d}$ comes with two natural morphisms. The first one is the {\em branch map}
\be
\beta: \Hb_{g,d} \rightarrow \overline{{\mathcal M}}_{0,n}\;,\label{bm}
\ee
that extends the natural covering $\H_{g,d} \rightarrow {\mathcal M}_{0,n}$ that maps $f$ to the configuration
$(z_1,\dots, z_n)$ of its ordered branch points considered up to the diagonal action of $PSL(2,\C)$.
The second one is the {\em forgetful map} 
\be
\pi:\Hb_{g,d} \rightarrow \overline{{\mathcal M}}_g\;,\label{fm}
\ee
that extends the natural projection $\H_{g,d} \rightarrow {\mathcal M}_g$ sending the equivalence class of the
branched cover $f:C\rightarrow \Pl$ to the isomorphism class of the covering curve $C$.

The description of the boundary $\Hb_{g,d}-\H_{g,d}$ is straightforward. Since we are interested only in the boundary
divisors, it is sufficient to consider admissible covers over the base $R$ consisting of two irreducible components
$R_1$ and $R_2$ intersecting at a single node $p$. The ramification type of the cover $f:C\rightarrow R$ over the node $p$
we will denote by $\mu=[m_1,\dots, m_r]$, where $r$ is the number of nodes of $C$ and $m_i$ is the ramification
index at the $i$-th node, $m_1+\dots +m_r=d$. Let us denote by $k$ and $n-k$ the number of branch points on
$R_1\setminus\{p\}$ and $R_2\setminus\{p\}$ respectively; we assume that $2\leq k\leq g+d-1$. Let $D_k$ be the divisor in 
$\overline{\mathcal M}_{0,n}$ parametrizing reducible curves with components of type $(0,k+1)$ and $(0,n-k+1)$,
and denote by $\Delta_k=\beta^{-1}(D_k)$ the preimage of $D_k$ in $\Hb_{g,d}$ with respect to the branch map (\ref{bm}).
The boundary divisor $\Delta_k$ is the union of divisors $\Delta_\mu^{(k)}$ over the set of all possible ramification types
$\mu$ over the node $p\in R$. Note that
the divisors $\Delta_\mu^{(k)}$ are generally reducible even for a fixed partition of branch points on $R$ and a fixed $\mu$.

The local structure of $\Hb_{g,d}$ near the divisors $\Delta_\mu^{(k)}$ was described in \cite{I}:
in the direction transversal to $\Delta_\mu^{(k)}$ with $\mu=[m_1,\dots,m_r],$ it looks like the (singular) curve
$$\zeta_1^{m_1}=\dots =\zeta_r^{m_r}$$
near the origin in $\C^r$. Therefore, for any irreducible component of $\Delta_\mu^{(k)}$ there are 
$\f{m_1\dots m_r}{m}$ (where $m={\rm l.c.m.}\{m_1,\dots,m_r\}$ is the least common multiple of $m_1,\dots,m_r$)
branches of $\Hb_{g,d}$ intersecting at it, whereas every such branch is an $m$-fold cover
of a neighborhood of $D_k$ in $\overline{\mathcal M}_{0,n}$ ramified over $D_k$ with ramification index $m$.

\subsection{Asymptotics of the tau function near the boundary}

Let $f:C\rightarrow\Pl$ be a holomorphic branched cover with only simple branch points $z_1,\dots, z_n\in\Pl,\;n=2g+2d-2$, and 
let $\gamma_i\mapsto s_i$ be the monodromy reprsentation, where $\gamma_i$ are non-intersecting simple loops about $z_i$
with some base point $z_0$,
and $s_1,\dots, s_n$ are transpositions in the symmetric group $S_d$ of $d$ elements such that $s_1\dots s_n=1$. 
Denote by $f_\e:C_\e\rightarrow\Pl$ the branched cover
with branch points $\e z_1,\dots, \e z_k, z_{k+1},\dots, z_n\in\Pl$ and the same monodromy as $f$, where we assume that 
$z_i\neq\infty$ for $i=1,\dots, k$ and $z_i\neq 0$ for $i=k+1,\dots, n$ ($2\leq k\leq g+d-1$ as above). At the limit $\e\rightarrow 0$ 
the map $f$ approaches an 
admissible cover $f_0: C_0\rightarrow R$, where $C_0$ is a genus $g$ nodal curve, and $R=\Pl_{(1)}\cup\Pl_{(2)}/\{\infty,0\}$ 
is the two component rational curve with one node $p=\{\infty, 0\}$ ($\infty\in\Pl_{(1)}$ is identified with $0\in\Pl_{(2)}$). The curve 
$C_0$ splits into two (not necessarily connected)
components $C_0^{(1)}$ and $C_0^{(2)}$ lying over $\Pl_{(1)}$ and $\Pl_{(2)}$ respectively. The restriction 
$f_0^{(1)}:C_0^{(1)}\rightarrow \Pl_{(1)}$ (resp. $f_0^{(2))}:C_0^{(2)}\rightarrow \Pl_{(2)}$) is simply branched over 
$z_1,\dots, z_k\in\Pl_{(1)}$
(resp. over $z_{k+1},\dots, z_n\in\Pl_{(2)}$).\footnote{This is because the functions $f_\e$ and $\e^{-1} f_\e$ represent the same point
in the Hurwitz space $\H_{g,n}$.} Moreover, $C_0^{(1)}$ (resp. $C_0^{(2)}$) is connected if and only if the group generated 
by $s_1,\dots, s_k$ (resp. by $s_{k+1},\dots, s_n$) acts transitively on the set of $d$ elements. The ramification type over the
node $p$ coincides with the type of the permutation $s_1\dots s_k\in S_d$ and, as above, we denote it by $\mu=[m_1,\dots, m_r]$.

We will need a canonical homology basis for the family of curves $C_\e$ that is compatible with the limiting nodal curve $C_0$.
Denote by $\ell$ the simple loop on $\Pl$ that shrinks to the node as $\e\rightarrow 0$, and by $\ell_1,\dots, \ell_r$
its preimages in $C_\e$ (we omit the dependence of these loops on the parameter $\e$). Choose some canonical bases 
$\alpha_1$ and $\alpha_2$ on the curves $C_0^{(1)}$ and $C_0^{(2)}$ respectively; we can pull them back to $C_\e$ in such 
a way, that they do not intersect the loops $\ell_1,\dots, \ell_r$. Denote by $[\ell_i]\in H^1(C_\e)$ the homology class of the loop
$\ell_i$, and put $q={\rm rank}\{[\ell_1],\dots, [\ell_r]\}$, that is, the rank of the linear span of the classes $[\ell_1],\dots, [\ell_r]$ in $H^1(C_\e)$.
An elementary topological consideration shows that $g=g_1+g_2+q$, where $g_1$ (resp. $g_2$) is the sum of genera of the
connected components of $C_0^{(1)}$ (resp. $C_0^{(2)}$). Without loss of generality, we can  
assume that $[\ell_1],\dots,[\ell_q]$ are linear
independent, and add $\ell_1,\dots,\ell_q$ to the union of $\alpha_1$ and $\alpha_2$ as $a$-cycles, while the corresponding 
$b$-cycles can be chosen as lifts of paths connecting branch points in different components of $\Pl-\ell$. 
We denote thus obtained basis on $C_\e$ by $\alpha$.

The main technical result of this paper is
\begin{theorem}
The isomonodromic tau function has the asymptotics
\be
\tau(C_\e^\alpha, f_\e)=\e^{3k-2d+2\sum_{i=1}^r 1/m_i}\; \tau(C_0^{(1),\alpha_1}, f_0^{(1)})\;
\tau(C_0^{(2),\alpha_2}, f_0^{(2)})\,(1+o(1)),
\la{astau}
\ee
as $\e\rightarrow 0$, where the tau function for a disconnected branched cover is understood as the product of tau functions for 
its connected components.
\end{theorem}

To prove this theorem we will need an auxiliary lemma. Together with $f_{\e}:C_{\e}^\alpha\rightarrow\Pl$
consider the branched cover $f_{\e}/\e:C_{\e}^\alpha\rightarrow\Pl$  with 
branch points $ z_1,\dots, z_k, \e^{-1}z_{k+1},\dots, \e^{-1}z_n\in\Pl$ and the same monodromy as $f$.  Denote the Bergman bidifferentials on the Torelli marked curves $C_{\e}^\alpha$, $C_0^{(1),\alpha_1}$ and $C_0^{(2),\alpha_2}$  by $B_{\e}$,
$B^{(1)}$ and $B^{(2)}$ respectively. 

We want to see what happens at the limit $\e\to 0$. 
We can always assume that  $|z_i|<1/\delta,\;i=1,\dots, k$, and 
$|z_i|>\delta,\;i=k+1,\dots, n$, for some $\delta\in (0,1)$. For small enough $\e$
consider two open subsets $D_\e^{(1)}= \left\{x\in C_\e\left| \,|f_\e(x)|<\e/\delta\right.\right\}$ and  
$D_\e^{(2)}= \left\{x\in C_\e\left| \,|f_\e(x)|>\delta\right.\right\}$ of the curve $C_\e$.
Note that the complement $C_\e-D_\e^{(1)}\cup D_\e^{(2)}$ is the union of $r$ disjoint cylinders
around the loops $\ell_1,\dots,\ell_r$. For each $\e$ the subset $D_\e^{(1)}$ (resp. $D_\e^{(2)}$)   
is naturally isomorphic to the subset 
$D_0^{(1)}=\left\{x\in C_0^{(1)}\left|\, |f_0^{(1)}(x)|<\f{1}{\delta}\right.\right\}$
(resp. $D_0^{(2)}=\left\{x\in C_0^{(2)}\left|\, |f_0^{(2)}(x)|>\delta\right.\right\}$).
As $\e \to 0$, we have
$$
f_\e(x)/\e\to f_0^{(1)}(x),\;\;\; x\in D_0^{(1)}
$$
and
$$
f_\e(x)\to f_0^{(2)}(x),\;\;\; x\in D_0^{(2)}\;.
$$

\begin{lemma}
In the limit $\e\to 0$ 
$$\f{B_\e(x,y)}{df_\e(x) \,df_\e(y)}\longrightarrow\f{B_{(1)}(x,y)}{df_0^{(1)}(x)\,df_0^{(1)}(y)}\;,\quad x,y\in D_0^{(1)},$$
and
$$\e^2\;\f{\,B_\e(x,y)}{df_\e(x) \,df_\e(y)}\longrightarrow\f{B_{(2)}(x,y)}{df_0^{(2)}(x)\,df_0^{(2)}(y)}\;,\quad x,y\in D_0^{(2)}$$
uniformly on $D_0^{(1)}$ and $D_0^{(2)}$ whenever $x\neq y$.
\end{lemma}

\begin{remark}
{\rm This lemma extends \cite{F}, Corollary 3.8, that treats the pinching of a single non-separating loop.}
\end{remark}

\begin{proof} 
According to our choice of the homology basis $\alpha$ on $C_\e$, the integrals of $B_\e$ along $a$-cycles coming from $C_0^{(1)}$ and $C_0^{(2)}$ are identically 0. Moreover, the integrals of $B_\e$ along the $r$ vanishing cycles 
$\ell_1,\dots, \ell_r$ tend to $0$ as $\e\to 0$. Therefore, repeating the argument of \cite{F}, Corollary 3.8, we see that 
the bidifferential $B_\e$ tends to $B^{(1)}$ on $D_0^{(1)}$ and to $B_0^{(2)}$ on $D_0^{(2)}$, as stated.
\end{proof}

Denote by $S_{B_\e}$, $S_{B^{(1)}}$ and $S_{B^{(2)}}$ the projective connections corresponding to the bidifferentials
$B_{\e}$, $B^{(1)}$ and $B^{(2)}$ respectively. From the above lemma we immediately get

\begin{corollary}\la{limSB}
The coefficients of the Bergman projective connection (\ref{Bercon}) have the following asymtotics as $\e\to 0$:
\begin{eqnarray}
\f{S_{B_\e}(x)- S_{f_\e}(x)}{df_\e(x)^2}&\longrightarrow &\f{S_{B^{(1)}}(x)- S_{f_0^{(1)}}(x)}{d f_0^{(1)}(x)^2}\;,\quad x\in D_0^{(1)}
\la{SB1}\\
\e^2\;\f{S_{B_\e}(x)- S_{f_\e/\e}(x)}{df_\e(x)^2}&\longrightarrow &
\f{S_{B^{(2)}}(x)- S_{f_0^{(2)}}(x)}{df_0^{(2)}(x)^2}\;,\quad\quad x\in D_0^{(2)}\;.
\la{SB2}
\end{eqnarray}
\end{corollary}

{\it Proof of Theorem 2.} 
Denote by $x_1^\e,\dots,x_n^\e\in C_\e$ the ramification points corresponding to the simple branch points
$\e z_1,\dots,\e z_k, z_{k+1},\dots, z_n\in\Pl$. 
By definition of $\tau(C_\e^\alpha, f_\e)$, cf. (\ref{tau}), we have
\begin{eqnarray}
\frac{\del}{\del (\e z_i)}\log\tau(C_\e^\alpha, f_\e)&=&  -4 \,{\rm Res}\,_{x_i^\e}\f{S_B^\e -S_{f_\e}}{df_\e}\;,\quad i=1,\dots, k,
\la{dertauxj}\\
\frac{\del}{\del z_i}\log\tau(C_\e^\alpha, f_\e)&=& -4 \,{\rm Res}\,_{x_i^\e}\f{S_B^\e -S_{f_\e}}{df_\e}\;,\quad i=k+1,\dots, n.
\la{dertauxj2}
\end{eqnarray}
From (\ref{dertauxj}) we see that for  $i=1,\dots, k$
$$
\frac{\del}{\del z_i}\log\tau(C_\e^\alpha, f_\e)=-4\,
{\rm Res}\,_{x_i^\e}\f{S_B^{\e} -S_{f^{\e}/\e}}{df_\e/\e}\;.
$$
Now  Corollary \ref{limSB} implies that, as $\e\to 0$,
\be
\tau(C_\e^\alpha, f_\e)=c(\e)
\; \tau(C_0^{(1),\alpha_1}, f_0^{(1)})\; \tau(C_0^{(2),\alpha_2}, f_0^{(2)})(1+o(1))
\ee
where 
$c(\e)$ is a function of $\e$ independent of $z_1,\dots,z_n$.
To explicitly compute $c(\e)$ we use Eq. (\ref{dertauxj}):
\be
\e\f{\del}{\del\e}\log\tau(C_\e^\alpha, f_\e)= -4\,\sum_{i=1}^{k}  z_i \, {\rm Res}\,_{x_i^\e}\f{S_B^\e -S_{f_\e}}{df_\e}\;.
\ee
From (\ref{SB1}) we get 
\be
\lim_{\e\to 0}\left(\e \f{\del}{\del\e}\log\tau(C_\e^\alpha, f_\e)\right)=
-4\sum_{i=1}^{k}  z_i\, {\rm Res}\,_{x_i}\f{S_B^{(1)} -S_{f_0^{(1)}}}{df_0^{(1)}}\;,
\ee
where the right-hand side is evaluated on the cover $f_0^{(1)}:C_0^{(1)}\rightarrow \Pl_{(1)}$.
Due to (\ref{Euler1}) we can rewrite the right hand side of the last formula in terms of  $k$, $d$ and the ramification type
$\mu=[m_1,\dots, m_r]$ over the node at $\infty\in\Pl_{(1)}$ as follows:
\be
\lim_{\e\to 0}\left(\e\f{\del}{\del\e}\log\tau(C_\e^\alpha, f_\e)\right) = 3 k - 2d + 2\sum_{i=1}^r \f{1}{m_i}.
\ee
Thus, 
$c(\e)=\e^{3k-2d-2\sum_{i=1}^r 1/m_i},$
which yields (\ref{astau}).

\begin{remark}
{\rm Asymptotic behavior of the tau function as $\epsilon \to 0$ can in principle be
derived from Theorem 2.4.13 and Eq.(2.4.9) of \cite{SMJ}, where it was described in
terms of traces of squares of the residues of the associated Fuchsian system in
a rather general situation. However, our approach is more straightforward and 
suits better for the situation we consider here.}
\end{remark}

\begin{corollary}
The (meromorphic) section ${\eta}$  of the line bundle $\lambda^{24(n-1)}$ on $\Hb_{g,d}$ (with $\lambda$ being the pullback
of the Hodge line bundle on $\overline{\mathcal M}_g$) has the following asymptotics as $\epsilon\to 0$:
\begin{eqnarray}
\eta(C_\e^\alpha, f_\e)&=&\e^{3 k (n-k)-2(n-1)(d-\sum_{i=1}^r {1/m_i})}\nonumber\\
&\times&{\eta}(C_0^{(1),\alpha_1}, f_0^{(1)})\; {\eta}(C_0^{(2),\alpha_2}, f_0^{(2)})(1+o(1)).\label{etaa}
\end{eqnarray}
\end{corollary}

\subsection{Relations between the divisors}

Here we discuss some explicit relations between the divisor classes in the rational Picard group 
${\rm Pic}(\Hb_{g,n})\otimes \Q$ that follow from 
the above analysis. Slightly abusing notation, we use the same symbols for line bundles and divisors on $\Hb_{g,d}$ as
for their classes in ${\rm Pic}(\Hb_{g,n})\otimes \Q$. It will also be convenient to understand the boundary divisors 
$\Delta_\mu^{(k)}$ in the orbifold sense, that is, as the weighted sums
of their irreducible components with weights $\f{1}{|{\rm Aut}(f)|}$, where ${\rm Aut}(f)$ is the automorphism group of a 
generic admissible cover $f$ parametrized by the irreducible component; such a ``weighted'' divisor we denote by 
$\delta_\mu^{(k)}$. Then we have

\begin{theorem}
For the Hodge class $\lambda\in {\rm Pic}(\Hb_{g,n})\otimes \Q$ the following formula holds:
\be
\lambda=\sum_{k=2}^{g+d-1}\sum_{\mu=[m_1,\dots, m_r]} \prod_{i=1}^r m_i
\left(\f{k(n-k)}{8(n-1)}-\f{1}{12}\left(d-\sum_{i=1}^r\f{1}{m_i}\right)\right)\delta_\mu^{(k)}\,.\label{main}
\ee
\end{theorem}

\begin{proof}
As it was mentioned in the end of Section 3.1, we can take $\e^{1/m},\,m={\rm l.c.m.}\{m_1,\dots,m_r\},$ as a transversal local parameter 
on each of the $\frac{m_1\dots m_r}{m}$ branches of $\Hb_{g,n}$ near each irreducible component of $\Delta_\mu^{(k)}$. 
Plugging it into (\ref{etaa}) and taking the action of ${\rm Aut}(f)$ into account, we prove the theorem.
\end{proof}

We finish with several comments concerning the special cases of the above theorem.

For $d=2$, Eq. (\ref{main}) takes the form
\be
\lambda=\sum_{i=1}^{[(g+1)/2]}\f{i(g+1-i)}{4g+2}\delta_{[1^2]}^{(2i)}+\sum_{j=1}^{[g/2]}\f{j(g-j)}{2g+1}\delta_{[2]}^{(2j+1)}\;.\label{CH}
\ee
This well-known formula expresses the Hodge class on the closure of the hyperelliptic locus in $\overline{{\mathcal M}}_g$ in terms 
of the boundary strata (cf. \cite{CH}, Proposition (4.7)). The only difference is that our coefficient at $\delta_{[1^2]}^{(2)}$ is
twice that of \cite{CH}. This is because the divisor $\delta_{[1^2]}^{(2)}$ parametrizes admissible covers containing
an irreducible genus 0 component with two nodes and two critical points that has a non-trivial automorphism group of order 2 
and gets contracted under the forgetful map
$\pi:\Hb_{g,2} \rightarrow \overline{{\mathcal M}}_g$. (In other words, we have $\delta_{[1^2]}^{(2)}=\f{1}{2}\pi^{-1}(\delta_0)$, where 
$\delta_0$ is the boundary divisor of irreducible curves in $\overline{{\mathcal M}}_g$.)

For $g=0$ one has $\lambda=0$, so that Eq. (\ref{main}) reads
\be
\sum_{k=2}^{d-1}\sum_{\mu=[m_1,\dots, m_r]}
\prod_{i=1}^r m_i\left(\f{k(2d-2-k)}{8(2d-3)}-\f{1}{12}\left(d-\sum_{i=1}^r\f{1}{m_i}\right)\right)\delta_\mu^{(k)}=0.\label{g=0}
\ee
Let us compare this formula with the results of \cite{LZ}. Put $\mathfrak{M}_{0,d}=\H_{0,d}/S_{2d-2}$, where the symmetric
group $S_{2d-2}$ acts by interchanging the $2d-2$ simple branch points, and denote by $\overline{\mathfrak{M}}_{0,d}$ the 
compactification of $\mathfrak{M}_{0,d}$ by means of the
stable maps. Consider the natural map 
$$\phi:\Hb_{0,d}\rightarrow\overline{\mathfrak{M}}_{0,d}\;,$$
and put 
$$C_d=\phi\left(\Delta_{[3\,1^{d-3}]}^{(2)}\right),\; M_d=\phi\left(\Delta_{[2^2\,1^{d-4}]}^{(2)}\right),\;
\Delta_d=\phi\left(\Delta_{[1^d]}^{(2)}\right).$$
The strata $C_d,M_d,\Delta_d$ are the divisors in $\overline{\mathfrak{M}}_{0,d}$, whereas $\phi(\Delta_\mu^{(k)})$ has codimension $\geq 2$ in
$\overline{\mathfrak{M}}_{0,d}$ for $k\geq 3$. According to \cite{LZ}, one has the relation
$$(d-6)C_d-3M_d+3(d-2)\Delta_d=0$$ 
in ${\rm Pic}(\overline{\mathfrak{M}}_{0,d})\otimes\Q$, and an easy check shows that this is consistent with (\ref{g=0}). 

For $g=1$ one has $\lambda=\f{1}{12}\{\infty\}$ on $\Mb_1$, where $\{\infty\}=\Mb_1-\M_1$ is the (one point) boundary divisor.
The preimage $\pi^{-1}(\{\infty\})\subset\Hb_{1,d}-\H_{1,d}$ with respect to the forgetful map (\ref{fm}) is the boundary divisor 
parametrizing nodal admissible covers with $g(C^{(1)})=g(C^{(2)})=0$. Therefore, (\ref{main}) gives a non-trivial relation between
the boundary divisors of $\Hb_{1,d}$. It would be instructive to compare this relation with the results of \cite{VZ}.

For $g=2$ one has $\lambda=\f{1}{10}\delta_0+\f{1}{5}\delta_1$ on $\Mb_2$, where $\delta_0$ (resp. $\delta_1$) is the divisor
of irreducible (resp. reducible) stable nodal curves (cf. \cite{Mu}). The preimage $\pi^{-1}(\delta_1)$ (resp. $\pi^{-1}(\delta_0)$)
in $\Hb_{2,d}-\H_{2,d}$ parametrizes admissible covers with $g(C^{(1)})=g(C^{(2)})=1$ (resp. with $g(C^{(1)})+g(C^{(2)})=1$, where 
the single irreducible genus 1 component intersects an irreducible genus 0 component at exactly two nodes). In this case we also have 
a nontrivial relation between the boundary divisors of $\Hb_{2,d}$.

{\bf Acknowledgements.} We started this work during our visit to the Max-Planck-Institut f\"ur Mathematik in Bonn, 
and gratefully acknowledge its hospitality and support; PZ also acknowledges his gratitude to IPMU, 
University of Tokyo at Kashiwa, where this work was finalized. DK thanks M.~Mazzocco, and PZ thanks M.~Kazarian, D.~Orlov 
and D.~Zvonkine for stimulating discussions. Our special thanks are to G.~van der Geer for pointing out a wrong sign
in the main formula (\ref{main}), and to the referee for the valuable remarks.

\end{document}